\documentclass[11pt]{amsart}

\usepackage{amsmath,amssymb,amsthm,mathtools}
\usepackage{enumitem}
\usepackage{hyperref}
\usepackage[margin=1.15in]{geometry}

\newtheorem{theorem}{Theorem}[section]
\newtheorem{lemma}[theorem]{Lemma}

\newtheorem{corollary}[theorem]{Corollary}

\newtheorem{problem}[theorem]{Problem}

\title{Rademacher Type and Enflo Type Revisited}

\author{Nika Areshidze}
\address{
Department of Mathematics,
University of California, Irvine,
CA 92697, USA
}
\email{nareshid@uci.edu}



\begin{document}

\begin{abstract}
We establish a dimension-free vector-valued
Poincar\'e-type inequality on the discrete cube for Banach spaces with
Rademacher type $p$, $1\le p\le2$. As a consequence, we obtain
$$
T_p^R(X)\le T_p^E(X)\le p\,T_p^R(X).
$$
This gives an alternative proof that Rademacher type and Enflo type coincide,
while improving the constant $\pi/\sqrt{2}$ in the quantitative estimate of
Ivanisvili, van Handel, and Volberg \cite{IVV} to $p$.
\end{abstract}

\maketitle

\vspace{-2.3em}

\section{Introduction}

The relation between Rademacher type and its metric analogue, Enflo type,
is a classical problem in the geometry of Banach spaces. In his 1978 paper
\cite{Enflo}, Enflo formulated a metric inequality on finite-dimensional
cubes and asked whether Rademacher type $p$ implies the corresponding
dimension-free cube inequality. Important partial results and related
notions were obtained in \cite{BMW,NaorSchechtman,MendelNaor}. Related
dimension-free forms of Pisier's inequality \cite{Pisier} were studied in
\cite{Eskenazis}. The problem was finally resolved by Ivanisvili, van Handel,
and Volberg \cite{IVV}, who proved that Rademacher type and Enflo type
coincide by establishing a dimension-free analogue of Pisier's inequality
on the discrete cube.

Let $(X,\|\cdot\|_{X})$ be a Banach space. We say that $X$ has Rademacher type
$p\in[1,2]$ if there exists $C\in(0,\infty)$ such that for every $n\ge1$,
every $x_1,\ldots,x_n\in X$, and independent Rademacher variables
$\varepsilon_1,\ldots,\varepsilon_n$,
$$
\mathbb E_{\varepsilon}
\left\|
\sum_{j=1}^n\varepsilon_jx_j
\right\|_X^p
\le
C^p\sum_{j=1}^n\|x_j\|_X^p.
$$
We denote by $T_p^R(X)$ the smallest possible constant $C$ in this
inequality.

We say that $X$ has Enflo type $p$ if there exists $C\in(0,\infty)$ such
that for every $n\ge1$ and every function $f:\{-1,1\}^n\to X$,
$$
\mathbb E_{\varepsilon}
\left\|
\frac{f(\varepsilon)-f(-\varepsilon)}{2}
\right\|_X^p
\le
C^p\sum_{j=1}^n
\mathbb E_{\varepsilon}\|D_jf(\varepsilon)\|_X^p.
$$
We denote by $T_p^E(X)$ the smallest possible constant $C$ in this
inequality.

Using random-reveal martingales, we give an alternative proof of the
equivalence between Rademacher type and Enflo type. The main ingredient is
the following dimension-free vector-valued Poincar\'e-type inequality.

\begin{theorem}
Let $X$ be a Banach space, let $1\le p\le2$, and suppose that $X$ has
Rademacher type $p$. Then for every $n\ge1$ and every $f:\Omega_n\to X$,
$$
\|f-\mathbb Ef\|_{L^p(\Omega_n,X)}
\le
p\,T_p^R(X)
\left(
\sum_{j=1}^{n}
\|D_jf\|_{L^p(\Omega_n,X)}^p
\right)^{1/p}.
$$
\end{theorem}

As a consequence, we obtain
$$
T_p^R(X)\le T_p^E(X)\le p\,T_p^R(X),
\qquad 1\le p\le2.
$$
The quantitative bound obtained in \cite{IVV} was
$$
T_p^E(X)\le \frac{\pi}{\sqrt{2}}\,T_p^R(X).
$$
Thus, since $p\le2<\pi/\sqrt{2}$, our argument improves the constant
$\pi/\sqrt{2}$ to $p$.

The paper is organized as follows. In Section~2 we introduce the notation
and the random-reveal martingales used in the proof. In Section~3 we prove
the dimension-free  vector-valued Poincar\'e-type inequality and deduce the equivalence between
Rademacher type and Enflo type.

\section{Preliminaries}

Let $(X,\|\cdot\|_X)$ be a Banach space and let $\Omega_n=\{-1,1\}^n$ denote the $n$-dimensional Hamming cube equipped with the uniform probability measure, and let $[n]=\{1,\ldots,n\}$. For a function $f:\Omega_n\to X$, we write
$$
\mathbb Ef
=
\frac{1}{2^n}\sum_{\varepsilon\in\Omega_n}f(\varepsilon),
$$
and, for $1\le p<\infty$,
$$
\|f\|_{L^p(\Omega_n,X)}
=
\left(
\mathbb E\|f(\varepsilon)\|_X^p
\right)^{1/p}.
$$
For $A\subseteq[n]$, we write
$\varepsilon_A=(\varepsilon_j)_{j\in A}$. Every function $f:\Omega_n\to X$ admits a Walsh expansion
$$
f(\varepsilon)
=
\sum_{S\subseteq[n]}\widehat f(S)W_S(\varepsilon),
$$
where
$$
W_S(\varepsilon)
=
\prod_{j\in S}\varepsilon_j,
\qquad
\widehat f(S)
=
\mathbb E\big[f(\varepsilon)W_S(\varepsilon)\big],
$$
with $W_\varnothing=1$. For $\varepsilon=(\varepsilon_1,\ldots,\varepsilon_n)\in\Omega_n$ and
$j\in[n]$, let
$$
\varepsilon^{(j)}
=
(\varepsilon_1,\ldots,-\varepsilon_j,\ldots,\varepsilon_n)
$$
be obtained from $\varepsilon$ by flipping its $j$-th coordinate. For a function $f:\Omega_n\to X$, define
$$
D_jf(\varepsilon)
=
\frac{f(\varepsilon)-f(\varepsilon^{(j)})}{2}.
$$

For a subset $A\subseteq[n]$, let
$$
\mathcal F_A
=
\sigma(\varepsilon_j:j\in A)
$$
be the $\sigma$-algebra generated by the coordinates indexed by $A$, and write
$$
E_Af
=
\mathbb E[f\mid\mathcal F_A]
$$
for the corresponding conditional expectation. Then clearly
$$
E_Af=\sum_{S\subseteq A}\widehat f(S)W_S(\varepsilon).
$$

If $A\subseteq[n]$ and $j\notin A$, we define the increment associated with revealing the $j$-th coordinate by
$$
d_j^Af
=
E_{A\cup\{j\}}f-E_Af.
$$

Finally, let $S_n$ denote the symmetric group on $[n]$, equipped with the uniform probability measure. Let $\pi=(\pi_1,\ldots,\pi_n)\in S_n$ be a uniformly distributed random permutation. For $0\le k\le n$, set
$$
A_k^\pi=\{\pi_1,\ldots,\pi_k\},
$$
with $A_0^\pi=\varnothing$, and define $M_k^\pi=E_{A_k^\pi}f$. Thus,
$$
M_0^\pi=\mathbb Ef,
\quad
M_n^\pi=f.
$$
Moreover, for every fixed $\pi\in S_n$, the sequence
$(M_k^\pi)_{k=0}^n$ is a martingale with respect to the filtration
$(\mathcal F_{A_k^\pi})_{k=0}^n$, and its martingale differences are given by
$$
M_k^\pi-M_{k-1}^\pi
=
d_{\pi_k}^{A_{k-1}^\pi}f.
$$

For $j\notin A$, the reveal increment and the discrete derivative satisfy
$$
d_j^Af=E_{A\cup\{j\}}D_jf.
$$
This follows from the Walsh--Fourier expansion. Let $1\le p<\infty$. From the identity above and the fact that conditional expectation is a contraction on $L^p$, we obtain
$$
\|d_j^Af\|_{L^p(\Omega_n,X)}
\le
\|D_jf\|_{L^p(\Omega_n,X)}.
$$

\begin{lemma}
Let $1\le p\le 2$, and suppose $X$ has Rademacher type $p$. Then for every $A\subseteq[n]$,
$$
\left\|\sum_{j\notin A}d_j^Af\right\|_{L^p(\Omega_n,X)}
\le
T_p^R(X)
\left(
\sum_{j\notin A}
\|d_j^Af\|_{L^p(\Omega_n,X)}^p
\right)^{1/p}.
$$
Consequently,
$$
\left\|\sum_{j\notin A}d_j^Af\right\|_{L^p(\Omega_n,X)}
\le
T_p^R(X)
\left(
\sum_{j\notin A}
\|D_jf\|_{L^p(\Omega_n,X)}^p
\right)^{1/p}.
$$
\end{lemma}

\begin{proof}
Let $j\notin A$. Notice that
$$
d_j^Af
=
\varepsilon_j
\sum_{S\subseteq A}
\widehat f(S\cup\{j\})W_S.
$$
Thus, we can write
$$
d_j^Af(\varepsilon)
=
\varepsilon_jg_j^A(\varepsilon_A),
$$
where
$$
g_j^A(\varepsilon_A)
=
\sum_{S\subseteq A}
\widehat f(S\cup\{j\})W_S(\varepsilon).
$$
For fixed $\varepsilon_A$, the variables $(\varepsilon_j)_{j\notin A}$ are independent Rademacher variables. Since $X$ has Rademacher type $p$,
$$
\mathbb E_{\varepsilon_{A^c}}
\left\|
\sum_{j\notin A}d_j^Af(\varepsilon)
\right\|_X^p
\le
T_p^R(X)^p
\sum_{j\notin A}
\|g_j^A(\varepsilon_A)\|_X^p.
$$
Averaging over $\varepsilon_A$, we obtain
$$
\mathbb E_{\varepsilon}
\left\|
\sum_{j\notin A}d_j^Af(\varepsilon)
\right\|_X^p
\le
T_p^R(X)^p
\sum_{j\notin A}
\mathbb E_{\varepsilon_A}
\|g_j^A(\varepsilon_A)\|_X^p.
$$
Since
$$
\|d_j^Af(\varepsilon)\|_X
=
\|g_j^A(\varepsilon_A)\|_X,
$$
and $d_j^Af$ does not depend on the remaining coordinates, taking the $p$-th root gives the first inequality. The second follows from
$$
\|d_j^Af\|_{L^p(\Omega_n,X)}
\le
\|D_jf\|_{L^p(\Omega_n,X)}.
$$
\end{proof}

\begin{lemma}
For every $1\le k\le n$,
$$
\mathbb{E}_{\pi}d_{\pi_k}^{A_{k-1}^{\pi}}f
=
\frac{1}{\binom{n}{k-1}(n-k+1)}
\sum_{\substack{A\subseteq[n]\\ |A|=k-1}}
\sum_{j\notin A}d_j^Af.
$$
\end{lemma}

\begin{proof}
Notice that
$$
\mathbb{E}_{\pi}d_{\pi_k}^{A_{k-1}^{\pi}}f
=
\frac{1}{n!}
\sum_{\substack{A\subseteq[n]\\ |A|=k-1}}
\sum_{j\notin A}
\sum_{\substack{\pi\in S_n\\
\{\pi_1,\ldots,\pi_{k-1}\}=A\\
\pi_k=j}}
d_j^Af.
$$
For fixed $A$ and $j\notin A$, there are $(k-1)!(n-k)!$ such permutations.
Therefore,
$$
\mathbb{E}_{\pi}d_{\pi_k}^{A_{k-1}^{\pi}}f
=
\frac{1}{\binom{n}{k-1}(n-k+1)}
\sum_{\substack{A\subseteq[n]\\ |A|=k-1}}
\sum_{j\notin A}d_j^Af.
$$
\end{proof}

\section{Proof of the Main Results}
\begin{proof}[Proof of Theorem~1.1]
For every $\pi\in S_n$,
$$
f-\mathbb{E}f
=
M_n^\pi-M_0^\pi
=
\sum_{k=1}^{n}d_{\pi_k}^{A_{k-1}^\pi}f.
$$
Averaging over $\pi$ and using the triangle inequality,
$$
\|f-\mathbb{E}f\|_{L^p(\Omega_n,X)}
\le
\sum_{k=1}^{n}
\left\|
\mathbb{E}_{\pi}d_{\pi_k}^{A_{k-1}^{\pi}}f
\right\|_{L^p(\Omega_n,X)}.
$$

By the preceding lemmas,
$$
\begin{aligned}
\left\|
\mathbb{E}_{\pi}d_{\pi_k}^{A_{k-1}^{\pi}}f
\right\|_{L^p(\Omega_n,X)}
&\le
\frac{T_p^R(X)}{n-k+1}
\frac{1}{\binom{n}{k-1}}
\sum_{\substack{A\subseteq[n]\\ |A|=k-1}}
\left(
\sum_{j\notin A}
\|D_jf\|_{L^p(\Omega_n,X)}^p
\right)^{1/p}.
\end{aligned}
$$
By Jensen's inequality,
$$
\begin{aligned}
\left\|
\mathbb{E}_{\pi}d_{\pi_k}^{A_{k-1}^{\pi}}f
\right\|_{L^p(\Omega_n,X)}
&\le
\frac{T_p^R(X)}{n-k+1}
\left(
\frac{1}{\binom{n}{k-1}}
\sum_{\substack{A\subseteq[n]\\ |A|=k-1}}
\sum_{j\notin A}
\|D_jf\|_{L^p(\Omega_n,X)}^p
\right)^{1/p}.
\end{aligned}
$$

For each fixed $j\in[n]$, there are $\binom{n-1}{k-1}$ sets
$A\subseteq[n]$ such that $|A|=k-1$ and $j\notin A$. Hence
$$
\left\|
\mathbb{E}_{\pi}d_{\pi_k}^{A_{k-1}^{\pi}}f
\right\|_{L^p(\Omega_n,X)}
\le
T_p^R(X)n^{-1/p}(n-k+1)^{1/p-1}
\left(
\sum_{j=1}^{n}
\|D_jf\|_{L^p(\Omega_n,X)}^p
\right)^{1/p}.
$$

Therefore,
$$
\begin{aligned}
\|f-\mathbb{E}f\|_{L^p(\Omega_n,X)}
&\le
T_p^R(X)n^{-1/p}
\left(
\sum_{k=1}^{n}(n-k+1)^{1/p-1}
\right)
\left(
\sum_{j=1}^{n}
\|D_jf\|_{L^p(\Omega_n,X)}^p
\right)^{1/p}
\\
&=
T_p^R(X)n^{-1/p}
\left(
\sum_{m=1}^{n}m^{1/p-1}
\right)
\left(
\sum_{j=1}^{n}
\|D_jf\|_{L^p(\Omega_n,X)}^p
\right)^{1/p}.
\end{aligned}
$$
Since $x^{1/p-1}$ is decreasing and integrable on $(0,n]$,
$$
\sum_{m=1}^{n}m^{1/p-1}
\le
\int_0^n x^{1/p-1}\,dx
=
pn^{1/p}.
$$
Consequently,
$$
\|f-\mathbb{E}f\|_{L^p(\Omega_n,X)}
\le
p\, T_p^R(X)
\left(
\sum_{j=1}^{n}
\|D_jf\|_{L^p(\Omega_n,X)}^p
\right)^{1/p}.
$$
\end{proof}

\begin{corollary}
Let $X$ be a Banach space and let $1\le p\le2$. Then
$$
T_p^R(X)\le T_p^E(X)\le p\,T_p^R(X).
$$
In particular, Rademacher type $p$ and Enflo type $p$ coincide.
\end{corollary}

\begin{proof}
The inequality $T_p^R(X)\le T_p^E(X)$ is immediate by applying the Enflo
type inequality to the linear function
$$
f(\varepsilon)=\sum_{j=1}^n\varepsilon_jx_j.
$$
It remains to prove the reverse estimate. For $f:\Omega_n\to X$, we have
$$
\frac{f(\varepsilon)-f(-\varepsilon)}{2}
=
\frac{(f(\varepsilon)-\mathbb Ef)-(f(-\varepsilon)-\mathbb Ef)}{2}.
$$
By convexity of $x\mapsto\|x\|_X^p$,
$$
\left\|
\frac{f(\varepsilon)-f(-\varepsilon)}{2}
\right\|_X^p
\le
\frac12\|f(\varepsilon)-\mathbb Ef\|_X^p
+
\frac12\|f(-\varepsilon)-\mathbb Ef\|_X^p.
$$
Averaging over $\varepsilon$ and using that $\varepsilon$ and
$-\varepsilon$ have the same distribution, we obtain
$$
\mathbb E_\varepsilon
\left\|
\frac{f(\varepsilon)-f(-\varepsilon)}{2}
\right\|_X^p
\le
\|f-\mathbb Ef\|_{L^p(\Omega_n,X)}^p.
$$
Applying Theorem~1.1 gives
$$
\mathbb E_\varepsilon
\left\|
\frac{f(\varepsilon)-f(-\varepsilon)}{2}
\right\|_X^p
\le
p^p\,T_p^R(X)^p
\sum_{j=1}^n
\|D_jf\|_{L^p(\Omega_n,X)}^p.
$$
Hence
$$
T_p^E(X)\le p\,T_p^R(X),
$$
which completes the proof.
\end{proof}
 
\begin{problem}
For $1< p\le2$, define
$$
C_p
=
\sup_X
\frac{T_p^E(X)}{T_p^R(X)},
$$
where the supremum is taken over all nonzero Banach spaces $X$ with Rademacher type $p$. By the preceding corollary,
$$
1\le C_p\le p.
$$
Is
$$
C_p=p?
$$
More generally, determine the optimal value of $C_p$.
\end{problem}

\section*{Acknowledgments}
The author acknowledges the use of AI tools. All mathematical arguments and proofs in the final manuscript were checked and written by the author.


\begin{thebibliography}{9}

\bibitem{BMW}
J.~Bourgain, V.~Milman, and H.~Wolfson,
\emph{On type of metric spaces},
Trans. Amer. Math. Soc. \textbf{294} (1986), no.~1, 295--317.

\bibitem{Enflo}
P.~Enflo,
\emph{On infinite-dimensional topological groups},
S\'eminaire sur la G\'eom\'etrie des Espaces de Banach
(1977--1978), Exp.~No.~10--11,
\'Ecole Polytechnique, Palaiseau, 1978, 11 pp.

\bibitem{Eskenazis}
A.~Eskenazis,
\emph{On Pisier's inequality for UMD targets},
Canad. Math. Bull. \textbf{64} (2021), no.~2, 282--291.

\bibitem{IVV}
P.~Ivanisvili, R.~van Handel, and A.~Volberg,
\emph{Rademacher type and Enflo type coincide},
Ann. of Math. (2) \textbf{192} (2020), no.~2, 665--678.

\bibitem{MendelNaor}
M.~Mendel and A.~Naor,
\emph{Scaled Enflo type is equivalent to Rademacher type},
Bull. Lond. Math. Soc. \textbf{39} (2007), no.~3, 493--498.

\bibitem{NaorSchechtman}
A.~Naor and G.~Schechtman,
\emph{Remarks on non linear type and Pisier's inequality},
J. Reine Angew. Math. \textbf{552} (2002), 213--236.

\bibitem{Pisier}
G.~Pisier,
\emph{Probabilistic methods in the geometry of Banach spaces},
in Probability and Analysis (Varenna, 1985),
Lecture Notes in Math. \textbf{1206},
Springer, Berlin, 1986, 167--241.

\end{thebibliography}
\end{document}